\documentclass[11 pt,english]{article}
\usepackage[margin = 1.9cm]{geometry}

\usepackage{amsthm}
\usepackage{amsmath}
\usepackage{amssymb}
\usepackage{mathtools}
\usepackage{graphicx}
\usepackage[pdftex,colorlinks,backref=page,citecolor=blue,bookmarks=false]{hyperref}
\usepackage{color}
\usepackage{enumitem}
\usepackage{comment}

\usepackage{setspace}
\usepackage{todonotes}

\usepackage{cleveref}
\theoremstyle{plain}

\newtheorem*{theorem*}{Theorem}
\newtheorem{theorem}{Theorem}[section]
\crefname{theorem}{Theorem}{Theorems}
\Crefname{theorem}{Theorem}{Theorems}

\newtheorem*{lemma*}{Lemma}
\newtheorem{lemma}[theorem]{Lemma}
\crefname{lemma}{Lemma}{Lemmas}
\Crefname{lemma}{Lemma}{Lemmas}

\newtheorem*{claim*}{Claim}
\newtheorem{claim}[theorem]{Claim}
\crefname{claim}{Claim}{Claims}
\Crefname{claim}{Claim}{Claims}

\newtheorem*{innerclaim*}{Claim}

\crefname{innerclaim}{Claim}{Claims}
\Crefname{innerclaim}{Claim}{Claims}

\newtheorem{proposition}[theorem]{Proposition}
\crefname{proposition}{Proposition}{Propositions}
\Crefname{proposition}{Proposition}{Propositions}

\newtheorem{remark}[theorem]{Remark}
\crefname{remark}{Remark}{Remarks}
\Crefname{remark}{Remark}{Remarks}

\newtheorem{corollary}[theorem]{Corollary}
\crefname{corollary}{Corollary}{Corollaries}
\Crefname{corollary}{Corollary}{Corollaries}

\newtheorem{conjecture}[theorem]{Conjecture}
\crefname{conjecture}{Conjecture}{Conjectures}
\Crefname{conjecture}{Conjecture}{Conjectures}

\crefname{question}{Question}{Questions}
\Crefname{question}{Question}{Questions}

\crefname{observation}{Observation}{Observations}
\Crefname{observation}{Observation}{Observations}

\crefname{example}{Example}{Examples}
\Crefname{example}{Example}{Examples}

\crefname{remark}{Remark}{Remarks}
\Crefname{remark}{Remark}{Remarks}

\theoremstyle{definition}

\crefname{problem}{Problem}{Problems}
\Crefname{problem}{Problem}{Problems}

\newtheorem{definition}[theorem]{Definition}
\crefname{definition}{Definition}{Definitions}
\Crefname{definition}{Definition}{Definitions}

\newcommand{\cF}{\mathcal{F}}

\newcommand{\cS}{\mathcal{S}}
\newcommand{\cP}{\mathcal{P}}
\newcommand{\cQ}{\mathcal{Q}}

\newcommand{\cG}{\mathcal{G}}
\newcommand{\cJ}{\mathcal{J}}

\newcommand{\eps}{\varepsilon}

\newcommand{\floor}[1]{\lfloor #1 \rfloor}

\begin{document}
\title{On a problem of Brown, Erd\H{o}s and S\'{o}s}
\author{
	Shoham Letzter\thanks{
		Department of Mathematics, 
		University College London, 
		Gower Street, London WC1E~6BT, UK. \\
		Email: \texttt{\{s.letzter|a.sgueglia\}}@\texttt{ucl.ac.uk}.
		This research is supported by the Royal Society.
	}
	\and
	Amedeo Sgueglia\footnotemark[1]
}

\maketitle

\begin{abstract}
	\setlength{\parskip}{\medskipamount}
	\setlength{\parindent}{0pt}
	\noindent
	Let $f^{(r)}(n;s,k)$ be the maximum number of edges in an $n$-vertex $r$-uniform hypergraph not containing a subhypergraph with $k$ edges on at most $s$ vertices.
	Recently, Delcourt and Postle, building on work of Glock, Joos, Kim, K\"{u}hn, Lichev and Pikhurko, proved that the limit $\lim_{n \to \infty} n^{-2} f^{(3)}(n;k+2,k)$ exists for all $k \ge 2$, solving an old problem of Brown, Erd\H{o}s and S\'{o}s (1973).
	Meanwhile, Shangguan and Tamo asked the more general question of determining if the limit $\lim_{n \to \infty} n^{-t} f^{(r)}(n;k(r-t)+t,k)$ exists for all $r>t\ge 2$ and $k \ge 2$.
	
	Here we make progress on their question. For every even $k$, we determine the value of the limit when $r$ is sufficiently large with respect to $k$ and $t$.
	Moreover, we show that the limit exists for $k \in \{5,7\}$ and all $r > t \ge 2$.

\end{abstract}

\section{Introduction}
\label{sec:intro}

	An $(s,k)$-configuration in an $r$-uniform hypergraph (henceforth \emph{$r$-graph}) is a collection of $k$ edges spanning at most $s$ vertices.
    Brown, Erd\H{o}s and S\'{o}s \cite{BES} started the investigation of the function $f^{(r)}(n;s,k)$, defined as the maximum number of edges in an $n$-vertex $r$-graph not containing an $(s,k)$-configuration.   
    In particular, they showed that $f^{(r)}(n; s, k) = \Omega(n^{(rk-s)/(k-1)})$ for all $s>r \ge 2$ and $k \ge 2$.
    Suppose now that the exponent $t:={(rk-s)/(k-1)}$ is an integer, so $s=k(r-t)+t$.
    Observe that $s$ is the number of vertices spanned by a $k$-edge $r$-graph where the edges can be ordered so that all but the first edge share exactly $t$ vertices with the previous edges.
    In particular, any set of vertices of size $t$ which is contained in $k$ distinct edges creates a $(k(r-t)+t, k)$-configuration.
    Therefore $f^{(r)}(n; k(r-t)+t, k) =O(n^t)$ and, with the above result of Brown, Erd\H{o}s and S\'os, we have
    \[
    	f^{(r)}(n; k(r-t)+t, k) = \Theta(n^t)\, .
    \]    
    A major open problem is the following conjecture, which was proposed by Shangguan and Tamo \cite{shanguann_tamo} and generalises an old conjecture of Brown, Erd\H{o}s and S\'{o}s \cite{BES} (corresponding to $r=3$ and $t=2$).

    \begin{conjecture}
    \label{conj:BES}
        For any positive integers $r, k, t$, the limit
        \[
       		\pi(r,t,k):= \lim\limits_{n \to \infty} n^{-t} f^{(r)}(n;k(r-t)+t,k)
       	\]
       	exists.
    \end{conjecture}

    We can assume $k \ge 2$ and $t \in [r-1]$, as otherwise \cref{conj:BES} is trivial. 
    Moreover, when $t=1$, it is easy to establish that the limit exists and that $\pi(r,1,k)=\frac{k-1}{(k-1)(r-1)+1}$, as already observed in \cite{glockJKKLP_BES_problem}.
    Indeed, the extremal constructions are vertex-disjoint unions of loose trees with $k-1$ edges, while the upper bound follows from the fact any collection of $k$ edges, which can be ordered so that all but the first edge share at least one vertex with the previous ones, is a $(k(r-1)+1,k)$-configuration.
    Therefore we can in fact assume $k \ge 2$ and $t \in [2,r-1]$, and we will do so in the rest of the paper.
    Recently, significant progress has been made towards \cref{conj:BES} and we now summarise the main developments.
    
    We start by discussing the results concerning the original conjecture of Brown, Erd\H{o}s and S\'{o}s, that is \cref{conj:BES} for $r=3$ and $t=2$, i.e.\ the existence of $\pi(3,2,k)$.
    Brown, Erd\H{o}s and S\'{o}s studied the case $k=2$ and  showed \cite{BES} that the limit is $\pi(3,2,2)=1/6$.
    More than $40$ years later, Glock \cite{glock_k=3} proved the conjecture for $k=3$ and determined that $\pi(3,2,3)=1/5$.
    Very recently, Glock, Joos, Kim, K\"{u}hn, Lichev and Pikhurko \cite{glockJKKLP_BES_problem} proved the conjecture for $k=4$ and determined that $\pi(3,2,4)=7/36$.
    In the concluding remarks, they also claim that their methods can be adapted to show that $\pi(3,2,k)=1/5$ for $k \in \{5,7\}$. 
    Finally, Delcourt and Postle \cite{delcourt_BES_problem} proved the Brown--Erd\H{o}s--S\'{o}s conjecture in full, i.e.\ they showed that $\pi(3,2,k)$ exists for all $k \ge 2$, although their method does not provide an explicit value for the limit.  

    Shanguann and Tamo \cite{shanguann_tamo} adapted the methods in \cite{glock_k=3} to any uniformity and showed that $\pi(r,2,3)=1/(r^2-r-1)$, and Shanguann \cite{shangguan2} adapted \cite{delcourt_BES_problem} to any uniformity and showed that $\pi(r,2,k)$ exists (but provided no explicit value).
    
    Concerning the general conjecture, the case $k=2$ follows from the work of R\"odl \cite{rodl} on the existence of asymptotic Steiner systems, and we have $\pi(r,t,2)=\frac{1}{t!}\binom{r}{t}^{-1}$.
    Glock, Joos, Kim, K\"{u}hn, Lichev and Pikhurko \cite{glockJKKLP_BES_problem} settled the cases $k=3$ and $k=4$ by showing that $\pi(r,t,3)=\frac{2}{t!} \left(2\binom{r}{t}-1\right)^{-1}$ for every $r \ge 2$ and $\pi(r,t,4)=\frac{1}{t!}\binom{r}{t}^{-1}$ for every $r \ge 4$ (note that the case $r=3$ (and $t = 2$) is covered by one of the results mentioned above and does not follow the same pattern).

    Our first result provides the exact value of the limit when $k$ is even and $r$ is sufficiently large in terms of $k$ and $t$.
    \begin{theorem}
    \label{thm:r>>t}
        Let $k$ be an even positive integer and $t \ge 2$ an integer.
        Then, for every integer $r$ satisfying $r \ge t+(k^3 \cdot t!)^{1/t}$, we have that $\lim\limits_{n \to \infty} n^{-t} f^{(r)}(n;k(r-t)+t,k)= \frac{1}{t!} \binom{r}{t}^{-1}$. 
    \end{theorem}
    We remark that, as mentioned above, this was already known for $k=2$ \cite{rodl} and $k=4$ \cite{glockJKKLP_BES_problem} (and all $r$).
    Moreover, it is interesting to observe that the behaviour for odd $k$ is potentially different.
    For example, from \cite{glockJKKLP_BES_problem}, it holds that $\pi(r,t,3) = \frac{2}{t!} \left(2\binom{r}{t}-1\right)^{-1}$. Therefore, we now focus on the case of $k$ being odd.
	
	Firstly, we completely settle \cref{conj:BES} for $k=5$.
    
    \begin{theorem}
    \label{thm:k=5}
        Let $r$, $t$ be integers satisfying $r > t \ge 2$.
        Then the limit $\lim\limits_{n \to \infty} n^{-t} f^{(r)}(n;5(r-t)+t,5)$ exists.
    \end{theorem}
    
	Finally, we settle \cref{conj:BES} for $k=7$. (Our proof does not work when $r = 3$ and $t = 2$, but this case has been resolved for all $k$ in \cite{delcourt_BES_problem}.)

    \begin{theorem}
    \label{thm:k=7}
        Let $r, t$ be integers satisfying $r > t \ge 2$ and $(r,t) \neq (3,2)$. Then the limit $\lim\limits_{n \to \infty} n^{-t} f^{(r)}(n;7(r-t)+t,7)$ exists.
    \end{theorem}

    More about the case of $k$ being odd can be found in the concluding remarks.

    \begin{remark}
    \label{remark:recent_paper}
        Shortly after this paper appeared on arXiv, Glock, Kim, Lichev, Pikhurko and Sun \cite{glockKLPS_recent} determined the value of $\pi(r,t,k)$ for $t=2$, $k \in \{5,6,7\}$ and all $r \ge 3$: they showed that for $k \in \{5,7\}$, it holds that $\pi(r,2,k)=1/(r^2-r-1)$ (observe this is the same value as for $k=3$), while for $k=6$ it holds that $\pi(3,2,6)=61/330$ and $\pi(r,2,6)=1/(r^2-r)$ for $r \ge 4$ (observe this extends \cref{thm:r>>t} when $t=2$ and $k=6$). 
    \end{remark}
    
    \paragraph{Organisation.}
    \Cref{sec:preliminaries} introduces the relevant notation and collects some preliminary results, including our key proposition (\cref{prop:density}) needed for the proofs of \cref{thm:k=5,thm:k=7}, which are proved in \Cref{sec:k=5} and \Cref{sec:k=7}, respectively. 
	\Cref{sec:even_k} provides the proof of \cref{thm:r>>t} and finally \Cref{sec:conclusion} contains some concluding remarks.
 
    \paragraph{Notation.}
            Given an $r$-graph $\cF$, we often think of $\cF$ as the edge set $E(\cF)$. In particular, by $|\cF|$ we mean the number of edges in $\cF$, and by $e \in \cF$ we mean that $e$ is an edge in $\cF$.
            Since the values of $r$ and $t$ will always be clear from the context, we introduce the following terminology.
            A \emph{$k$-configuration} denotes a $(k(r-t)+t,k)$-configuration, while a \emph{$k^-$-configuration} denotes a $(k(r-t)+t-1,k)$-configuration.
            Moreover, we say that a hypergraph is \emph{$k$-free} (resp.\ \emph{$k^-$-free}) if it does not contain any $k$-configuration (resp.\ $k^-$-configuration). 

\section{Preliminaries} \label{sec:preliminaries}
	
	\subsection{Lower bounds}
    In order to build $k$-free $r$-graphs with many edges, the strategy of Glock, Joos, Kim, K\"{u}hn, Lichev and Pikhurko \cite{glockJKKLP_BES_problem} consisted of packing many copies of a carefully chosen $k$-free $r$-graph of constant size, while making sure not to create any $k$-configurations using edges from different copies.
    Before stating their main technical result, we introduce some definitions.

	Recall that the \emph{$t$-shadow} of a hypergraph $\cF$, denoted $\partial_t\cF$, is the $t$-graph on $V(\cF)$ whose edges are the $t$-subsets of edges in $\cF$.
      
	\begin{definition} \label{def:supporting-graph}
		Given an $r$-graph $\cF$ and a $t$-graph $J$, we say that $J$ is a \emph{supporting $t$-graph} of $\cF$ if $V(J)=V(\cF)$ and $J$ contains the $t$-shadow of $\cF$.
		For such $\cF$ and $J$, we define the \emph{non-edge girth} of $(\cF,J)$ to be the smallest $g\ge 1$ for which there exists a $g$-configuration in $\cF$ whose vertex set contains a non-edge of~$J$. 
		Equivalently, it is the largest $g \ge 1$ such that for every $\ell$-configuration $S$ in $\cF$ with $\ell < g$, all $t$-subsets of $V(S)$ are edges of $J$.
		If no such $g$ exists, we set the non-edge girth of $(\cF,J)$ to be infinity.
	\end{definition}
    
	Here is the main technical result in \cite{glockJKKLP_BES_problem}.
    
    \begin{theorem}[Theorem 3.1 in \cite{glockJKKLP_BES_problem}]
    \label{thm:glockJKKLP}
        Fix $k\ge 2$, $r\ge 3$ and $t\in [2,r-1]$. Let $\cF$ be an $r$-graph which is $k$-free and $\ell^-$-free for all $\ell\in [2,k-1]$. Let $J$ be a supporting $t$-graph of $\cF$ such that the non-edge girth of $(\cF,J)$ is greater than $k/2$. Then, 
        \[
           \liminf_{n\to \infty} n^{-t}f^{(r)}(n;k(r-t)+t,k) \ge \frac{|\cF|}{t!\,|J|}\, .
        \]
    \end{theorem}

    In particular, by choosing $\cF$ to be a single $r$-uniform edge and $J=\binom{V(\cF)}{t}$, the hypotheses of \cref{thm:glockJKKLP} hold and we get the following corollary.

    \begin{corollary}[Corollary 3.2 in \cite{glockJKKLP_BES_problem}]
    \label{cor:glockJKKLP}
        Fix $k\ge 2$, $r\ge 3$ and $t\in [2,r-1]$. 
        Then,
        \[
           \liminf_{n\to \infty} n^{-t}f^{(r)}(n;k(r-t)+t,k) \ge \frac{1}{t!\,\binom{r}{t}}\, .
        \]        
    \end{corollary}

	\subsection{Density argument}
	The approach of Delcourt and Postle \cite{delcourt_BES_problem}, while proving \cref{conj:BES} for $r=3$, $t=2$ and any $k \ge 2$, relies on the following reduction: they show that in any sufficiently dense $k$-free $3$-graph, it is possible to find a subgraph with almost the same density which is additionally $\ell^-$-free for every $\ell \in [2,k-1]$, i.e.\ having $\ell^-$-configurations is `inefficient' for the extremal $k$-free graph.
	Here we provide another density-type argument, which we will use in the proof of \cref{thm:k=5,thm:k=7}.
	Before stating the result, we introduce some notation.
	Given an $r$-graph $\cF$, define $J(\cF)$ to be the $t$-graph with $V(\cF)$ as vertex set and where a $t$-subset $T \subseteq V(\cF)$ is an edge of $J(\cF)$ if and only if there exists an $\ell$-configuration for some $\ell \in \left[ \lfloor k/2 \rfloor \right]$ whose vertex set contains $T$.
	Observe that, since every edge is a $1$-configuration, $J(\cF)$ contains the $t$-shadow of $\cF$. 
	Therefore $J(\cF)$ is a supporting $t$-graph of $\cF$ (recall \Cref{def:supporting-graph}).
	Moreover, its non-edge girth is greater than $\lfloor k/2 \rfloor$.

	\begin{proposition}
		\label{prop:density}
		Suppose that for every $\eps > 0$ and large enough $n$, for every $k$-free $n$-vertex $r$-graph $\cF$ with $|\cF| \ge \big(\binom{r}{t}^{-1} + \eps\big) \binom{n}{t}$ there exist subhypergraphs $\cF_2 \subseteq \cF_1 \subseteq \cF$ such that $|\cF_1| \ge |\cF|-O(n^{t-1})$, $\cF_2$ is $\ell^-$-free for every $\ell \in [2,k-1]$, and 
		\begin{equation}
			\label{eq:density_argument}
			\frac{|\cF_2|}{|J(\cF_2)|} \ge \frac{|\cF_1|}{|J(\cF_1)|}\, . 
		\end{equation}
		Then the limit $\lim\limits_{n \to \infty} n^{-t} f^{(r)}(n;k(r-t)+t,k)$ exists.
	\end{proposition}
	
	\begin{proof}
		Define $\alpha$ to satisfy $\frac{\alpha}{t!} = \limsup\limits_{n \to \infty} n^{-t} f^{(r)}(n;k(r-t)+t,k)$ and observe that, since \cref{cor:glockJKKLP} gives that $\liminf\limits_{n \to \infty} n^{-t} f^{(r)}(n;k(r-t)+t,k) \ge \frac{1}{t!} \binom{r}{t}^{-1}$, it holds that $\alpha \ge \binom{r}{t}^{-1}$.
		If we have equality, we are done.
		Therefore, we can assume the inequality is strict and thus for small enough $\eps>0$ we have $\alpha > \binom{r}{t}^{-1} + \eps$.
		Given the definition of $\alpha$, for every $n \in \mathbb{N}$, there exist $m \ge n$ and an $m$-vertex $k$-free $r$-graph $\cF$ with $|\cF| \ge (\alpha - \eps) \binom{m}{t}$.
		Owing to the assumptions of the proposition, there exist subhypergraphs $\cF_2 \subseteq \cF_1 \subseteq \cF$ such that $|\cF_1| \ge |\cF|-O(m^{t-1})$, $\cF_2$ is $\ell^-$-free for every $\ell \in [2,k-1]$, and $\frac{|\cF_2|}{|J(\cF_2)|} \ge \frac{|\cF_1|}{|J(\cF_1)|}$.
		As observed above, $J(\cF_2)$ is a supporting $t$-graph of $\cF_2$ and its non-edge girth is greater than $k/2$.
		Therefore, by \cref{thm:glockJKKLP}, we have
		\begin{align*}
			\liminf_{n\to \infty} n^{-t}f^{(r)}(n;k(r-t)+t,k) 
			\ge \frac{|\cF_2|}{t! \cdot |J(\cF_2)|} 
			& \ge \frac{|\cF_1|}{t! \cdot |J(\cF_1)|} \\
			& \ge \frac{(\alpha-\eps)\binom{m}{t}-O(m^{t-1})}{t! \cdot \binom{m}{t}} 
			= \frac{\alpha - \eps}{t!} - O(m^{-1})\, ,
		\end{align*}
		using that $|J(\cF_1)| \le \binom{m}{t}$ for the last inequality.
		Since $\eps$ can be made arbitrarily small and $m$ arbitrarily large, the conclusion easily follows from
		\begin{equation*}
			\liminf_{n\to \infty} n^{-t}f^{(r)}(n;k(r-t)+t,k)  
			\ge \frac{\alpha}{t!} = \limsup_{n\to \infty} n^{-t}f^{(r)}(n;k(r-t)+t,k) \, .
			\qedhere
		\end{equation*}
	\end{proof}

	\begin{remark} \label{rem:equivalent-condition}
		We remark that if $\cF_1$ is an $r$-graph with $|\cF_1| \ge \binom{r}{t}^{-1} |J(\cF_1)|$ and $\cF_2 \subseteq \cF_1$, then the condition \begin{equation}
			\label{eq:density_argument_second_condition}
			|J(\cF_1)| - |J(\cF_2)| \ge \binom{r}{t} \left(|\cF_1|-|\cF_2|\right)
		\end{equation}
		implies Condition \eqref{eq:density_argument}. 
		Indeed, writing $x_1=|J(\cF_1)|$, $y_1=|\cF_1|$, $x_2=|J(\cF_2)|$, $y_2=|\cF_2|$ and $\alpha=\binom{r}{t}$, we have $x_1 \le \alpha y_1$, $x_2 \le x_1$ and $y_2 \le y_1$ by assumption.
		Moreover, by \eqref{eq:density_argument_second_condition}, $x_2 \le x_1 - \alpha(y_1 - y_2) \le \alpha y_2$.
		Therefore, using \eqref{eq:density_argument_second_condition} again, which is equivalent to $\alpha y_1 - x_1 \le \alpha y_2 - x_2$, together with $x_2 \le \alpha y_2$ and $x_2 \le x_1$, we have
		\begin{align*}
			\alpha (x_2y_1-x_1y_2) 
			& = x_2(\alpha y_1 - x_1) +x_1x_2-\alpha x_1 y_2 \\
			& \le x_2(\alpha y_2 -x_2) + x_1(x_2-\alpha y_2) 
			= (x_1-x_2)(x_2-\alpha y_2) \le 0.
		\end{align*}
		This implies $x_1y_2 \ge x_2y_1$, which in turn is equivalent to \eqref{eq:density_argument}.
	\end{remark}

	\subsection{A useful lemma}
	In order to apply the density argument of \cref{prop:density}, for any given $k$-free $n$-vertex $r$-graph $\cF$, we need to find a subhypergraph which is $\ell^-$-free for each $\ell \in [2,k-1]$ and satisfies some additional properties.
	It turns out that, for some values of $\ell$, there is a simple argument which shows that $\cF$ can be made $\ell^-$-free by removing only $O(n^{t-1})$ edges.
	This is established, together with additional properties, by the following lemma.
	
	\begin{lemma}
		\label{lemma:cleaning}
		Let $r,k$ and $t$ be fixed positive integers.
		Let $\cF$ be a $k$-free $n$-vertex $r$-graph.
		Then there exists a subhypergraph $\cF'$ of $\cF$ such that the following holds.
		\begin{enumerate}[label=\upshape(P\arabic*)]
			\item \label{cleaning:P1} 
				$\cF'$ is $\ell^-$-free for every $\ell \in [2,k]$ with $\ell | (k-1)$ or $\ell |k$;
			\item \label{cleaning:P4}
				there is no $3^-$-configuration in $\cF'$ which contains a $2$-configuration;
			\item \label{cleaning:P2} 
				for every positive integers $a$ and $b$ with $a + b =k$, every $a^-$-configuration and every $b$-configuration of $\cF'$ are edge-disjoint;
			\item \label{cleaning:P3} $|\cF'| \ge |\cF| - O(n^{t-1})$.
		\end{enumerate}
	\end{lemma}

	\begin{proof}
		We show that we can get a subhypergraph of $\cF$ which satisfies \ref{cleaning:P1}, \ref{cleaning:P4} and \ref{cleaning:P2} by removing $O(n^{t-1})$ edges, which in turn will imply \ref{cleaning:P3} as well.
		
		Observe that, since $\cF$ is $k$-free, $\cF$ is also $k^-$-free.
		Moreover we show that $\cF$ can be made $(k-1)^-$-free by removing $O(n^{t-1})$ edges.
		Let $\cS$ be a maximal collection of pairwise edge-disjoint $(k-1)^-$-configurations of $\cF$.
		If $|\cS| > \binom{n}{t-1}$, then there exists a set $T \subseteq V(\cF)$ of size $t-1$ which is contained in the $(t-1)$-shadow of two $(k-1)^-$-configurations $S_1$ and $S_2$ in $\cS$.
		Let $e \in S_2$ be an edge such that $T \subseteq e$.
		Then $S_1 \cup \{e\}$ is a $k$-configuration in $\cF$, being a collection of $k$ edges spanning at most $[(k-1)(r-t)+t-1]+(r-|T|)=k(r-t)+t$ vertices, a contradiction to $\cF$ being $k$-free.
		Therefore $|\cS| \le \binom{n}{t-1}$ and, by removing from $\cF$ all edges of each $S \in \cS$, we obtain a subhypergraph $\cF_0 \subseteq \cF$ which is $(k-1)^-$-free and satisfies $|\cF_0| \ge |\cF| - O(n^{t-1})$.

		Let $2 \le \ell < k-1$ with $\ell | (k-1)$ (resp.\ $\ell | k$).        
		Let $j > 1$ be the positive integer such that $\ell \cdot j = k-1$ (resp.\ $\ell \cdot j = k$).
		Let $\cS_{\ell}$ be a maximal collection of pairwise edge-disjoint $\ell^-$-configurations in $\cF_0$.
		If $|\cS_{\ell}| >(j-1) \cdot \binom{n}{t-1}$, then there exists a set $T \subseteq V(\cF)$ of size $t-1$ which is contained in the vertex set of $j$ distinct $\ell^-$-configurations $S_1, \dots, S_j$ in $\cS_{\ell}$.
		Then $S_1 \cup \dots \cup S_j$ is a $(k-1)^-$-configuration of $\cF_0$, being a collection of $\ell \cdot j = k-1$ edges spanning at most $j[\ell(r-t)+t-1]-(j-1)|T|=j\ell(r-t)+t-1=(k-1)(r-t)+t-1$ vertices (resp.\ a $k^-$-configuration).
		This is a contradiction to $\cF_0$ being $(k-1)^-$-free (resp.\ $k^-$-free).
		Therefore $|\cS_{\ell}| = O(n^{t-1})$ for all relevant $\ell$.
		
		If $k \equiv \text{$0$ or $1$} \pmod 3$, then \ref{cleaning:P1} would trivially imply \ref{cleaning:P4}, as $\cF'$ would not contain any $3^-$-configurations, and we set $\cS':=\emptyset$. 
		If that is not the case, namely if $k \equiv 2 \pmod 3$, define $\cS'$ to be a maximal collection of pairwise edge-disjoint $3^-$-configurations containing a $2$-configuration. We claim that $|\cS'| \le \frac{k-2}{3} \cdot \binom{n}{t-1}$. Indeed, otherwise, there is a $(t-1)$-subset $T \subseteq V(\cF)$ and $(k-2)/3+1=(k+1)/3$ many $3^-$-configurations $S_1, \ldots, S_{(k+1)/3} \in \cS'$, where $S_i$ contains a $2$-configuration $S_i'$ satisfying $T \subseteq V(S_i')$. Then $S_1 \cup \ldots \cup S_{(k-2)/3} \cup S_{(k+1)/3}'$ is a $k$-configuration, being a collection of $k$ edges spanning at most $((k-2)/3) \cdot [3(r-t) + t-1] + [2(r-t) + t] - ((k-2)/3) \cdot (t-1) = k(r-t) + t$, a contradiction to $\cF_0$ being $k$-free.
		Therefore $|\cS'|=O(n^{t-1})$. 
		
		Let $a$ and $b$ be positive integers with $a+b=k$.
		Let $H_a$ be the collection of edges contained in $a^-$-configurations of $\cF_0$ and let $\cS_{a,b}$ be a maximal collection of pairwise edge-disjoint $b$-configurations of $\cF_0$ containing an edge of $H_a$.
		If $|\cS_{a,b}| > a \cdot \binom{n}{t-1}$, then there exists a set $T \subseteq V(\cF)$ of size $t-1$ and $a+1$ distinct $b$-configurations $S_1, \dots, S_{a+1}$ in $\cS_{a,b}$ such that there exists $e_i \in S_i \cap H_a$ with $T \subseteq e_i$ for every $i \in [a+1]$.
		By definition of $H_a$, there exists $f_2,\dots,f_a \in \cF_0$ such that $S':=\{e_1,f_2,\dots,f_a\}$ is an $a^-$-configuration and, without loss of generality, we assume that $S_{a+1}$ and $S'$ are edge-disjoint.
		Then $S_{a+1} \cup S'$ is a $k$-configuration of $\cF_2$, being a collection of $b+a=k$ edges spanning at most $[b(r-t)+t]+[a(r-t)+t-1]-|T|=k(r-t)+t$, a contradiction to $\cF_0$ being $k$-free.
		Therefore $|\cS_{a,b}|=O(n^{t-1})$.
		
		Let $\cF'$ be the subhypergraph of $\cF_0$ which is obtained by removing all the edges of each $S \in \cS_{\ell}$ for every $2 \le \ell < k-1$ with $\ell|(k-1)$ or $\ell | k$, all the edges of each $S \in \cS'$, and all the edges of each $S \in \cS_{a,b}$ for every positive integers $a$ and $b$ with $a+b=k$.
		Then $\cF'$ satisfies \ref{cleaning:P1}, \ref{cleaning:P4}, \ref{cleaning:P2} and \ref{cleaning:P3}.
	\end{proof}

	We remark that \cref{prop:density} and \Cref{lemma:cleaning} offer short proofs that \cref{conj:BES} holds for $k=2$ and $k=3$. 
	Indeed, the case $k=2$ is immediate. 
	For $k=3$, given a $3$-free $r$-graph $\cF$, \Cref{lemma:cleaning} gives a subhypergraph $\cF' \subseteq \cF$ which is $2^-$-free and satisfies $|\cF'| \ge |\cF|-O(n^{t-1})$.
	We can then apply \cref{prop:density} with $\cF_1 = \cF_2=\cF'$. 

\section{Proof of \cref{thm:r>>t} (\cref{conj:BES} for $k$ even)} \label{sec:even_k}   
    In this section, we prove \cref{thm:r>>t}, which asserts that $\lim_{n \to \infty}n^{-t}f^{(r)}(n; k(r-t)+t,k)=\frac{1}{t!}\binom{r}{t}^{-1}$, for $k$ even and $r$ sufficiently large in terms of $t$ and $k$.
    We do that by showing that $\liminf_{n\to \infty} n^{-t}f^{(r)}(n;k(r-t)+t,k) \ge \frac{1}{t! \, \binom{r}{t}}$, which follows directly from \cref{cor:glockJKKLP}, and that $\limsup_{n\to \infty} n^{-t}f^{(r)}(n;k(r-t)+t,k) \le \frac{1}{t! \, \binom{r}{t}}$, which follows from \cref{lemma:cleaning} and the inequality and claims provided below. 

	\begin{claim} \label{claim:calc}
		Suppose that $r,t,k$ are integers satisfying $t,k \ge 2$ and $r \ge t + (k^3\cdot t!)^{1/t}$.
		Then
		\begin{equation*}
			\binom{2r-t}{t} - \left[2\binom{r}{t} - 1\right] - (k-3) \ge (k-2)^3.
		\end{equation*}
	\end{claim}
	\begin{proof}
		Notice that
		\begin{align*}
			\binom{2r-t}{t} &= \frac{1}{t!} \cdot \prod_{i=0}^{t-1} (2r-t-i) = \frac{1}{t!} \cdot  (2r-2t+1) (2r-2t+2) \cdot \prod_{i=0}^{t-3} (r-i+r-t)\\
			& \ge \frac{1}{t!} \cdot 2 \cdot (r-t+1)(r-t+2) \cdot \left[ \prod_{i=0}^{t-3} (r-i) + \prod_{i=0}^{t-3} (r-t)\right] \\
			&\ge 2 \binom{r}{t} + \frac{2}{t!}(r-t)^t\, ,
		\end{align*}
		where in the second line we used that $2 \le t \le r-1$. This in turn gives
		\[
		\binom{2r-t}{t}-2\binom{r}{t} 
		\ge \frac{2}{t!}\cdot(r-t)^t 
		\ge k^3
		\ge (k-2)^3+(k-3)\, ,
		\] 
		where the second inequality follows from $r \ge t + (k^3\cdot t!)^{1/t}$.    	
	\end{proof}

    \begin{claim}
    \label{claim:e_in_few_2_config}
        Let $\cF$ be a $k$-free $r$-graph and $e \in \cF$.
        Then the number of $2$-configurations of $\cF$ containing $e$ is at most $k-2$.    
    \end{claim}

    \begin{proof}
        If there were $k-1$ distinct $2$-configurations $\{e,e_i\}$ for $i \in [k-1]$, then $S:=\{e,e_1,\dots,e_{k-1}\}$ would be a $k$-configuration of $\cF$, being a collection of $k$ edges spanning at most $r+(k-1)(r-t)=k(r-t)+t$ vertices, a contradiction to $\cF$ being $k$-free.
    \end{proof}

    \begin{claim}
    \label{claim:T_in_few_2_config}
        Let $k$ be an even integer, $\cF$ a $k$-free $r$-graph and $T \subseteq V(\cF)$ with $|T| = t$.
        Then the number of $2$-configurations whose vertex set contains $T$ is at most $(k-2)^2$.    
    \end{claim}

    \begin{proof}
    	Let $\cS$ be a maximal collection of pairwise edge-disjoint $2$-configurations of $\cF$ whose vertex set contains $T$.
    	We have $|\cS| \le (k-2)/2$, as otherwise there would exist distinct $2$-configurations $S_1,\dots,S_{k/2}$ in $\cS$ and $S_1 \cup \dots \cup S_{k/2}$ would give a $k$-configuration of $\cF$, being a collection of $k$ edges spanning at most $(2r-t)k/2-(k/2-1)|T|=k(r-t)+t$ vertices.
    	By maximality of $\cS$, any $2$-configuration of $\cF$ whose vertex set contains $T$, must contain an edge which belongs to some $S \in \cS$.
    	There are $2|\cS| \le k-2$ such edges and, for each of them, by \cref{claim:e_in_few_2_config}, the number of $2$-configurations of $\cF$ containing this edge is at most $k-2$.
    	Therefore the number of $2$-configurations of $\cF$ whose vertex set contains $T$ is at most $(k-2)^2$.
    \end{proof}

    Given a hypergraph $\cF$, we say that a $t$-set $T$ of $V(\cF)$ is \emph{covered} exactly $i$ times, if $T$ is contained in exactly $i$ edges of $\cF$.
    We denote by $J_i(\cF)$ the set of $t$-subsets of $V(\cF)$ covered exactly $i$ times, and by $J_{\ge i}(\cF)$ the set of $t$-subsets of $V(\cF)$ covered at least $i$ times.
    \begin{claim}
    \label{claim:J0_J2}
        Let $k$ be a positive even integer and $\cF$ be an $r$-graph which is $k$-free, $2^-$-free, and has no $3^-$-configurations containing a $2$-configuration. Write $J_0 := J_0(\cF)$ and $J_{\ge 2} := J_{\ge 2}(\cF)$.
        Then 
        \[
           (k-2)^2|J_0| \ge \left\{ \binom{2r-t}{t}-\left[2\binom{r}{t}-1\right] - (k-3) \right\} \cdot |J_{\ge 2}| \, .  
        \]
    \end{claim}

    \begin{proof}
        We use a double counting argument on the set
        \[
        \cQ:=\left\{(S,T):~
        \begin{array}{c}
             \text{$S$ is a $2$-configuration of $\cF$,} \\
             \text{$T \subseteq V(S)$, $|T|=t$ and $T \in J_0$}
        \end{array}
        \right\}
        \, .
        \]
        
        Fix $T \subseteq V(\cF)$ with $|T|=t$.
        By \cref{claim:T_in_few_2_config}, the number of $2$-configurations of $\cF$ whose vertex set contains $T$ is at most $(k-2)^2$.
        We conclude that
        \begin{equation}
        \label{eq:Q_upper_bound}
            |\cQ| \le (k-2)^2|J_0| \, .    
        \end{equation}

        Now fix a $2$-configuration $S:=\{f_1,f_2\}$ and observe that, since $\cF$ is $2^-$-free, $S$ spans precisely $2r-t$ vertices, so $f_1$ and $f_2$ share precisely $t$ vertices.
        We now estimate the number of $t$-sets $T \subseteq V(S)$ with $T \in J_0$.
        Since $T \subseteq V(S)$, either $T$ is fully contained in $f_1$ or in $f_2$, or intersects both $f_1 \setminus f_2$ and $f_2 \setminus f_1$.
        
        If $T$ is fully contained in $f_1$, then it does not belong to $J_0$, as it is covered at least once (by the edge $f_1$).
        Clearly, the same argument applies to any $T$ which is fully contained in $f_2$.
        Moreover, the number of such $t$-sets is $2\binom{r}{t}-1$.
        
        Now we consider those $T$ intersecting both $f_1 \setminus f_2$ and $f_2 \setminus f_1$.
        If $T \not\in J_0$, then there exists $e \in \cF$ with $T \subseteq e$, and clearly $e \neq f_1, f_2$.
		Notice that, $|e \cap V(S)| \le t$, by the assumption that there are no $3^-$-configurations containing a $2$-configuration. As $T \subseteq e \cap V(S)$, we have that $e \cap V(S) = T$.
        It follows that among the $t$-sets of $V(S)$ intersecting both $f_1 \setminus f_2$ and $f_2 \setminus f_1$, all but at most $k-3$ belong to $J_0$. 
        Indeed, otherwise, there would exist pairwise distinct $t$-sets $T_1, \dots, T_{k-2} \subseteq V(S)$ and pairwise distinct edges $e_1, \dots, e_{k-2}$ with $T_i \subseteq e_i$ for each $i \in [k-2]$.
        However, $\{e_1,\dots,e_{k-2},f_1,f_2\}$ would be a $k$-configuration, being a collection of $k$ edges spanning at most $(k-2)(r-t)+2r-t=k(r-t)+t$ vertices.
        
        Therefore, for a given $2$-configuration $S$, the number of $t$-sets $T \subseteq V(S)$ with $T \in J_0$ is at least
        \begin{equation}
        \label{eq:T_subsets}
            \binom{2r-t}{t}-\left[2\binom{r}{t}-1\right] - (k-3)\, ,
        \end{equation}
        where the first term stands for the number of $t$-sets of $V(S)$, while the rest accounts for the arguments above.

		Finally, observe that every $T \in J_{\ge 2}$ gives rise to a $2$-configuration $\{f_1, f_2\}$ with $T = f_1 \cap f_2$ (we have $T \subseteq f_1 \cap f_2$ by definition, with equality because $\cF$ is $2^-$-free), and these $2$-configurations are distinct for different sets $T$. This shows that the number of $2$-configurations of $\cF$ is at least $|J_{\ge 2}|$.
        Using \eqref{eq:T_subsets}, we conclude that 
        \begin{equation}
        \label{eq:Q_lower_bound}
            |\cQ| \ge \left\{ \binom{2r-t}{t}-\left[2\binom{r}{t}-1\right] - (k-3) \right\} \cdot |J_{\ge 2}| \, .        
        \end{equation}
        
        The claim follows from \eqref{eq:Q_upper_bound} and \eqref{eq:Q_lower_bound}.
    \end{proof}

 	We are now ready to prove \cref{thm:r>>t}.
    
    \begin{proof}[Proof of \cref{thm:r>>t}]
        Let $k$ be an even integer, $t \ge 2$ an integer and let $r$ be an integer satisfying $r \ge t + (k^3 \cdot t!)^{1/t}$.
        From \cref{cor:glockJKKLP} we get
        \begin{equation}
        \label{eq:liminf}
           \liminf_{n\to \infty} n^{-t}f^{(r)}(n;k(r-t)+t,k) \ge \frac{1}{t! \, \binom{r}{t}}\, .
        \end{equation}

        Let $\cF$ be a $k$-free $n$-vertex $r$-graph.
        By \cref{lemma:cleaning}, there exists a subhypergraph $\cF'$ of $\cF$ which is $2^-$-free, has no $3^-$-configurations containing a $2$-configuration and satisfies $|\cF'| \ge |\cF| - O(n^{t-1})$.
		Set $J_i:=J_i(\cF')$ and $J_{\ge i}:=J_{\ge i}(\cF')$ and observe that applications of \cref{claim:J0_J2} and \Cref{claim:calc} give 
        \[
            (k-2)^2|J_0| 
			\ge \left\{ \binom{2r-t}{t}-\left[2\binom{r}{t}-1\right] - (k-3) \right\} \cdot |J_{\ge 2}|
			\ge (k-2)^3|J_{\ge 2}|\,.
        \]
    	Therefore, $|J_0| \ge (k-2) \cdot |J_{\ge 2}|$.
        Now consider the following set
        \[
            \cP:=\{(e,T): e \in \cF', T \subseteq e \text{ with } |T|=t\}\, .
        \]
        Then $|\cP|=|\cF'| \cdot \binom{r}{t}$ and
        \begin{align*}
            |\cP|=\sum_{i \ge 1} i \cdot |J_i| & \le |J_1| + (k-1) \cdot |J_{\ge 2}| \\
            & = |J_{\ge 0}| + (k-2) \cdot |J_{\ge 2}|-|J_0| \le \binom{n}{t}\, , 
        \end{align*}
        where in the first inequality we used that a $t$-set covered at least $k$ times gives a $k$-configuration and thus, since $\cF'$ is $k$-free, we have $J_{\ge k} = \emptyset$, while in the last inequality we used $|J_{\ge 0}| = \binom{n}{t}$ and $|J_0| \ge (k-2) \cdot |J_{\ge 2}|$.
        We conclude that 
        \[
            |\cF| 
			\le |\cF'| + O(n^{t-1}) 
			= |\cP| \cdot \binom{r}{t}^{-1} + O(n^{t-1})
			\le \binom{n}{t} \cdot \binom{r}{t}^{-1} + O(n^{t-1})\, , 
        \]
        which allows us to establish that
        \begin{equation}
        \label{eq:limsup}
            \limsup_{n\to \infty} n^{-t}f^{(r)}(n;k(r-t)+t,k) \le \frac{1}{t!\, \binom{r}{t}}\, .           
        \end{equation}
    The theorem follows from \eqref{eq:liminf} and \eqref{eq:limsup}.
    \end{proof}

\section{Proof of \cref{thm:k=5} (\cref{conj:BES} for $k=5$)} \label{sec:k=5}

	In this section we prove \Cref{thm:k=5}, asserting that the limit $\lim_{n \to \infty}n^{-t}f^{(r)}(n; 5(r-t)+5)$ exists, for $2 \le t < r$.
	Our proof uses our density argument (\cref{prop:density}).
	We recall that $J(\cF)$ is the $t$-graph on $V(\cF)$ whose edges are $t$-subsets of $\ell$-configurations in $\cF$ with $\ell \le \floor{k/2}$ (this is defined above \Cref{prop:density}).

    \begin{proof}[Proof of \cref{thm:k=5}]

        Let $\eps > 0$ and $\cF$ be a $5$-free $n$-vertex $r$-graph with $|\cF| \ge \big(\binom{r}{t}^{-1}+\eps\big) \binom{n}{t}$, and suppose that $n$ is large.
        By \cref{lemma:cleaning}, there is a subhypergraph $\cF_1 \subseteq \cF$ which is $2^-$-free, $4^-$-free and $5$-free, satisfies 
		\begin{equation} \label{eqn:F1}
			|\cF_1| \ge |\cF| - O(n^{t-1}) \ge \binom{r}{t}^{-1} \binom{n}{t} \ge \binom{r}{t}^{-1}|J(\cF_1)|,
		\end{equation}
		and where any $2$-configuration and any $3^-$-configuration are edge-disjoint.

        \begin{claim} \label{claim:k=5_claim3}
			Let $\cG \subseteq \cF_1$ and suppose that $S$ is a $3^-$-configuration in $\cG$. Then the following holds with $\cG' := \cG \setminus S$.
			\begin{equation*}
				|J(\cG)| - |J(\cG')| \ge \binom{r}{t}\left(|\cG| - |\cG'|\right)\,.
			\end{equation*}
        \end{claim}

        \begin{proof} 
        	Observe that, by definition of $J(S)$, we have that $T \in J(S)$ if and only there exists an edge $e \in S$ with $T \subseteq e$ or there exists a $2$-configuration $S'$ in $S$ whose vertex set contains $T$.
        	Since in $\cG$ any $2$-configuration and any $3^-$-configuration are edge-disjoint, we can rule out the second option.
        	Moreover, a set of size $t$ cannot be in more than one edge of $S$ as, otherwise, $S$ would contain a $2$-configuration, which cannot happen for the same reason.
        	Since $S$ has three edges, we conclude that $|J(S)|=3\binom{r}{t}$.
        	
            Let $T \in J(S)$.
            Then clearly $T \in J(\cG)$, and we aim to show that $T \not\in J(\cG')$.
            For that, note that $T \in J(\cG')$ if and only if there exists an edge $e \in \cG'$ with $T \subseteq e$ or there exists a $2$-configuration $S'$ of $\cG'$ whose vertex set contains $T$.
            The first option cannot happen as, otherwise, $S \cup \{e\}$ would be a $4^-$-configuration of $\cG$, being a collection of four edges spanning at most $(3r-2t-1)+r-t=4(r-t)+t-1$ vertices, a contradiction to $\cF_1$ being $4^-$-free.  
            Similarly, we can rule out the second option as, otherwise, using that $S$ and $S'$ are edge-disjoint, $S \cup S'$ would be a $5^-$-configuration of $\cG$, being a collection of five edges spanning at most $(3r-2t-1)+(2r-t)-t=5(r-t)+t-1$ vertices, a contradiction to $\cF_1$ being $5^-$-free.
            
            Therefore $|J(\cG)| - |J(\cG')| \ge |J(S)| = 3 \binom{r}{t}$ and the claim follows.
		\end{proof}
	
		By applying \cref{claim:k=5_claim3} repeatedly, we can find a subhypergraph $\cF_2 \subseteq \cF_1$, which is $3^-$-free and satisfies
		\begin{equation*}
			|J(\cF_{1})| - |J(\cF_2)| \ge \binom{r}{t} \left(|\cF_{1}| - |\cF_2|\right)\, .
		\end{equation*}
		By \Cref{rem:equivalent-condition} and \eqref{eqn:F1}, it follows that
		\begin{equation*}
			\frac{|\cF_2|}{|J(\cF_2)|} \ge \frac{|\cF_{1}|}{|J(\cF_{1})|}\, .
		\end{equation*}
        Notice that $\cF_2$ is $\ell^-$-free for $\ell \in \{2,3,4\}$ and $5$-free.
        Thus \Cref{prop:density} implies that the limit $\lim_{n \to \infty}n^{-t}f^{(r)}(n;5(r-t)+t,5)$ exists. 
    \end{proof}

\section{Proof of \cref{thm:k=7} (\cref{conj:BES} for $k=7$)} \label{sec:k=7}

This section is concerned with the proof of \Cref{thm:k=7}, which asserts that the limit $\lim_{n \to \infty}n^{-t}f^{(r)}(n;7(r-t)+t,7)$ exists for $r > t \ge 2$ and $(r,t) \neq (3,2)$.
We use our density argument (\cref{prop:density}), together with the following two inequalities.

\begin{claim} \label{claim:calc-2}
	Let $r,t$ be integers such that $3 \le t < r$ or $t = 2$ and $r \ge 4$. Then
	\begin{equation*}
		\binom{3r-2t}{t} - 4 \ge 3 \binom{r}{t}\,.
	\end{equation*}
\end{claim}

\begin{proof}
	The claimed inequality can be checked directly for $t = 2$, so suppose that $t \ge 3$.
	\begin{align*}
		\binom{3r-2t}{t} &= \frac{1}{t!} \cdot \prod_{i=0}^{t-1} (3r-2t-i) \\
		& = \frac{1}{t!} \cdot (3r-3t+1) (3r-3t+2) (3r-3t+3) \cdot \prod_{i=0}^{t-4} (r-i+2r-2t)\\
		& \ge \frac{1}{t!} \cdot 4(r-t+1)(r-t+2)(r-t+3) \cdot \prod_{i = 0}^{t-4}(r-i) \\
		&= 4\binom{r}{t} \ge 3\binom{r}{t} + 4\,.
	\end{align*}
	Here in the first inequality we used the inequality $(3x+1)(3x+2)(3x+3) \ge 4(x+1)(x+2)(x+3)$ for $x \ge 1$, which can be checked directly.
	In the last inequality we used that $\binom{r}{t} \ge \binom{t+1}{t} = t+1 \ge 4$.
\end{proof}

\begin{claim} \label{claim:calc-3}
	Let $r,t$ be integers such that $3 \le t < r$ or $t = 2$ and $r \ge 4$.
	Then 
	\begin{equation*}
		\binom{2r-t}{t} \ge 2\binom{r}{t} + 2.
	\end{equation*}
\end{claim}

\begin{proof}
	Let $A, B, C$ be pairwise disjoint sets of sizes $r-t, r-t, t$, respectively.
	Then $\binom{2r-t}{t}$ is the number of $t$-subsets of $A \cup B \cup C$.
	This is at least the number of $t$-subsets of either $A \cup C$ or $B \cup C$, of which there are $2\binom{r}{t} - 1$, plus the number of $t$-subsets of $A \cup B \cup C$ consisting of one vertex from each of $A$ and $B$ and $t-2$ vertices from $C$, of which there are $(r-t)^2 \binom{t}{t-2} \ge 3$, using that either $t \ge 3$ or $r-t \ge 2$. Altogether, we have that $\binom{2r-t}{t} \ge 2\binom{r}{t} - 1 + 3 = 2\binom{r}{t}+2$, as required.
\end{proof}

We are now ready to prove \Cref{thm:k=7}.

\begin{proof}[Proof of \cref{thm:k=7}]

    Let $\eps > 0$ and $\cF$ be a $7$-free $n$-vertex $r$-graph with $|\cF| \ge \big( \binom{r}{t}^{-1} + \eps\big) \binom{n}{t}$, and suppose that $n$ is large.

    Apply \cref{lemma:cleaning} to get a subhypergraph $\cF_1 \subseteq \cF$ which is $2^-$-free, $3^-$-free, $6^-$-free and $7$-free, satisfies 
	\begin{equation} \label{eqn:F11}
		|\cF_1| \ge |\cF| - O(n^{t-1})
		\ge \binom{r}{t}^{-1}\binom{n}{t} 
		\ge \binom{r}{t}^{-1} |J(\cF_1)|,
	\end{equation}
	and where any $2$-configuration and any $5^-$-configuration are edge-disjoint, and any $3$-configuration and any $4^-$-configuration are edge-disjoint.
    Now we prove some structural claims on subhypergraphs of $\cF_1$.
    
    \begin{claim} \label{claim:K=7_claim1}
		Let $\cG \subseteq \cF_1$ and suppose that $S$ is a $3$-configuration in $\cG$ contained in a $4$-configuration in $\cG$.
		Then the following holds with $\cG' := \cG \setminus S$.
        \[
			|J(\cG)| - |J(\cG')| \ge \binom{r}{t} \left(|\cG| - |\cG'|\right) \, .
        \]
    \end{claim}

    \begin{proof}
		Write $S := \{e_1, e_2, e_3\}$ and let $e_4 \in \cG$ be such that $\{e_1, e_2, e_3, e_4\}$ is a $4$-configuration.
        Let $T':=V(S) \cap  e_4$ and observe that $|T'|=t$.
        Indeed, $|T'| \ge t$ follows from the fact that $S$ is a $3$-configuration but not a $3^-$-configuration, implying that $|V(S)|=3r-2t$, and $S \cup \{e_4\}$ is a $4$-configuration, while $|T'| \le t$ follows from the fact that otherwise $S \cup \{e_4\}$ would be a $4^-$-configuration, a contradiction to any $3$-configuration and any $4^-$-configuration of $\cF_1$ being edge-disjoint.
        
        We now lower bound $|J(\cG)| - |J(\cG')|$.
        Let $T \subseteq V(S)$ satisfy $|T|=t$.
        Since $S$ is a $3$-configuration, we have $T \in J(\cG)$.
		For $T$ to be in $J(\cG')$ there must be an $\ell$-configuration in $\cG'$ with $\ell \in [3]$ whose vertex set contains $T$.
		We now prove the following assertions, to help us bound the number of times each of these options can happen.
        \begin{enumerate}[label=(\roman*)]
            \item \label{K=7_i} excluding $T'$, no $t$-subset of $V(S)$ is contained in the vertex set of a $3$-configuration of $\cG'$;
            \item \label{K=7_ii} at most one $t$-subset of $V(S)$ is contained in the vertex set of a $2$-configuration of $\cG'$; 
            \item \label{K=7_iii} excluding $T'$, at most two $t$-subsets of $V(S)$ are contained in an edge of $\cG'$.
        \end{enumerate}
        We show \ref{K=7_i} as follows.
        Any $3$-configuration $S'$ in $\cG'$ satisfies $|V(S) \cap V(S')| \le t$ as otherwise $S \cup S'$ would be a $6^-$-configuration.
        Therefore if a $t$-subset $T \neq T'$ of $V(S)$ is contained in a $3$-configuration $S'$ in $\cG'$, then $e_4 \not\in S'$ and $S \cup S' \cup \{e_4\}$ is a $7$-configuration of $\cG$, a contradiction.

        For \ref{K=7_ii} we argue as follows.
        First observe that any $2$-configuration $S'$ in $\cG'$ satisfies $|V(S) \cap V(S')| \le t$ as otherwise $S \cup S'$ would be a $5^-$-configuration of $\cG$, which is a contradiction as any $5^-$-configuration and any $2$-configuration of $\cG$ are edge-disjoint.
		Now suppose there are two distinct $t$-subsets $T_1$ and $T_2$ of $V(S)$ and two $2$-configurations $S_1$ and $S_2$ in $\cG'$ with $T_i \subseteq V(S_i)$ for $i \in [2]$.
        Observe that $S_1 \neq S_2$ as any $2$-configuration in $\cG'$ intersects $V(S)$ in no more than $t$ vertices, as argued above.
        If $S_1$ and $S_2$ were edge-disjoint, then $S_1 \cup S_2 \cup S$ would be a $7$-configuration of $\cG$, a contradiction.
        If $S_1$ and $S_2$ were not edge-disjoint, then $S_1 \cup S_2$ would be a $3$-configuration intersecting $V(S)$ in more than $t$ vertices, but then $S_1 \cup S_2 \cup S$ would be a $6^-$-configuration of $\cG$, a contradiction.

        Finally, we prove \ref{K=7_iii}.
        Any edge not in $S$ intersects $V(S)$ in at most $t$ vertices, as $3$-configurations and $4^-$-configurations of $\cG$ are edge-disjoint.
		Suppose there were three distinct $t$-subsets $T_1,T_2$ and $T_3$ of $V(S)$, all distinct from $T'$, and three (distinct) edges $f_1,f_2$ and $f_3$ not in $S$ with $T_i \subseteq f_i$ for $i \in [3]$.
        Then $e_4 \not\in \{f_1,f_2,f_3\}$ and $S \cup \{e_4,f_1,f_2,f_3\}$ would be a $7$-configuration of $\cG$, a contradiction.

        Recall that $|V(S)|=3r-2t$, so there are $\binom{3r-2t}{t}$ subsets of $V(S)$ of size $t$.
		Taking \ref{K=7_i}, \ref{K=7_ii} and \ref{K=7_iii} into account, and using \Cref{claim:calc-2}, we get
		\[
			|J(\cG)| - |J(\cG')| \ge \binom{3r-2t}{t} - 1 - 3 \ge 3 \binom{r}{t} = \binom{r}{t}(|\cG|-|\cG'|)\, ,
		\]
		as claimed.
    \end{proof}

    By repeatedly applying \cref{claim:K=7_claim1}, we get a subhypergraph $\cF_2 \subseteq \cF_1$ satisfying 
	\begin{equation} \label{eqn:F2}
		|J(\cF_1)| - |J(\cF_2)| \ge \binom{r}{t}\left(|\cF_1| - |\cF_2|\right),
	\end{equation}
	which has no $3$-configuration contained in a $4$-configuration.
    
    \begin{claim} \label{claim:K=7_claim2}
		Let $\cG \subseteq \cF_2$.
		Suppose that $S$ is a $4^-$-configuration in $\cG$. Then there exists a non-empty subset $S' \subseteq S$, such that the following holds with $\cG' := \cG \setminus S'$.
		\begin{equation} \label{eqn:JG}
			|J(\cG)| - |J(\cG')| \ge \binom{r}{t} \left(|\cG| - |\cG'|\right) \, .
		\end{equation}
    \end{claim}

    \begin{proof}
        We start by observing that, for $e,e' \in \cG$, if $e \in S$ and $\{e,e'\}$ is a $2$-configuration, then $e' \in S$. Indeed, otherwise, the $5^-$-configuration $S \cup \{e'\}$ and the $2$-configuration $\{e,e'\}$ would not be edge-disjoint, a contradiction. Therefore, either $S$ contains a $2$-configuration or the edges of $S$ are not involved in any $2$-configuration of $\cG$.
		Since $S$ contains no $3$-configurations (by every $4^-$-configuration being edge-disjoint of all $3$-configurations in $\cG$), we have the following three cases: $S$ contains no $2$-configurations; $S$ contains a single $2$-configuration; and $S$ can be partitioned into two $2$-configurations.
		We consider each case separately.

	\paragraph{Case 1. $S$ contains no $2$-configurations.}

		Let $e$ be an edge of $S$, and set $\cS':=\{e\}$ and $\cG':=\cG \setminus \{e\}$.
		We claim that $J(\cG) \setminus J(\cG')$ contains all $t$-subsets of $V(S')$, which would prove \eqref{eqn:JG} in this case.
        Since any such $t$-subset belongs to $J(\cG)$, this follows once we show that
        \begin{enumerate}[label=(\roman*)]
            \item \label{K=7_second_claim_i'} no $t$-subset of $e$ is contained in the vertex set of a $3$-configuration of $\cG'$;
            \item \label{K=7_second_claim_ii'} no $t$-subset of $e$ is contained in the vertex set of a $2$-configuration of $\cG'$;
            \item \label{K=7_second_claim_iii'} no $t$-subset of $e$ is contained in an edge of $\cG'$.
        \end{enumerate}
        Fact \ref{K=7_second_claim_i'} holds since, by assumption, any $3$-configuration is edge-disjoint of $S$ and thus, if it shares $t$ vertices with $e$, its union with $S$ would give a $7^-$-configuration, a contradiction.
        For \ref{K=7_second_claim_ii'}, recall that if there was a $t$-subset of $e$ contained in a $2$-configuration $S''$ then $e \not\in S''$ and $e$ would belong to both the $3$-configuration $S'' \cup \{e\}$ and the $4^-$-configuration $S$, a contradiction. 
        Finally, \ref{K=7_second_claim_iii} holds as otherwise $e$ would belong to a $2$-configuration of $\cG$, a contradiction to the assumption that $S$ contains no $2$-configurations and its edges are thus not involved in $2$-configurations in $\cG$.

	\paragraph{Case 2. $S$ contains a single $2$-configuration $S'$.}
        
        Set $\cG':=\cG \setminus S'$.
        Let $T \subseteq V(S')$ satisfy $|T|=t$.
        Since $S'$ is a $2$-configuration, we have $T \in J(\cG)$.
		For $T$ to be in $J(\cG')$, there must be an $\ell$-configuration in $\cG'$ with $\ell \in [3]$ whose vertex set contains $T$.
		We prove the following assertions, to help us bound the number of times this can happen.
        \begin{enumerate}[label=(\roman*)]
            \item \label{K=7_second_claim_i} no $t$-subset of $V(S')$ is contained in the vertex set of a $3$-configuration of $\cG'$;
            \item \label{K=7_second_claim_ii} no $t$-subset of $V(S')$ is contained in the vertex set of a $2$-configuration of $\cG'$;
            \item \label{K=7_second_claim_iii} no $t$-subset of $V(S')$ is contained in an edge of $\cG'$.
        \end{enumerate}
        Indeed, \ref{K=7_second_claim_i} can be proved as in the previous case. 
        For \ref{K=7_second_claim_ii}, if $S''$ is a $2$-configuration of $\cG'$ which intersects $V(S')$ in (at least) $t$ vertices then, by the assumption on $S$, the configurations $S$ and $S''$ are edge-disjoint, but then $S \cup S''$ is a $6^-$-configuration in $\cG$, a contradiction.
		Finally, \ref{K=7_second_claim_iii} holds since $e \in \cG'$ intersects $V(S')$ in at most $t-1$ vertices, as otherwise $S$ and $S' \cup \{e\}$ are $4^-$- and $3$-configurations that are not edge-disjoint.

		By \ref{K=7_second_claim_i}, \ref{K=7_second_claim_ii}, \ref{K=7_second_claim_iii} and \Cref{claim:calc-3}, we have
		\begin{equation*}
			|J(\cG)| - |J(\cG')| 
			\ge \binom{2r-t}{t} 
			\ge 2\binom{r}{t} 
			= \binom{r}{t} \cdot \left(|\cG| - |\cG'|\right).
		\end{equation*}

	\paragraph{Case 3. $S$ can be partitioned into two $2$-configurations $S_1$, $S_2$.}

		Set $S':=S$ and $\cG' := \cG \setminus S'$.
		Let $\cJ$ be the collection of $t$-sets which are subsets of either $V(S_1)$ or $V(S_2)$.
		Note that if $T$ is in the $t$-shadow of $S_1$ then $T$ is not a subset of $V(S_2)$ (otherwise, $S$ would contain a $3$-configuration). 
        Thus,
		\begin{equation} \label{eqn:J}
			|\cJ| \ge 
			|\partial_t S_1| + \binom{|V(S_2)|}{t}
			= 2\binom{r}{t} - 1 + \binom{2r-t}{t}
			\ge 4\binom{r}{t}\, ,
		\end{equation}
		using \Cref{claim:calc-3}.
		Notice that $\cJ \subseteq J(\cG)$, as its elements are $t$-subsets of vertex sets of $2$-configurations.
		As usual, we claim that
        \begin{enumerate}[label=(\roman*)]
            \item \label{K=7_second_claim_i''} no $t$-set in $\cJ$ is contained in the vertex set of a $3$-configuration of $\cG'$;
            \item \label{K=7_second_claim_ii''} no $t$-set in $\cJ$ is contained in the vertex set of a $2$-configuration of $\cG'$;
            \item \label{K=7_second_claim_iii''} no $t$-set in $\cJ$ is contained in an edge of $\cG'$.
        \end{enumerate}

		Assertion \ref{K=7_second_claim_i''} can be proved as in the first case.
		For \ref{K=7_second_claim_ii''}, if $S''$ is a $2$-configuration in $\cG'$ whose vertex set contains a $t$-set in $\cJ$, then $S \cup S''$ is a $6^-$-configuration in $\cG$, a contradiction.
		Finally, for \ref{K=7_second_claim_iii''}, if $e$ is an edge containing a $t$-set in $\cJ$ then $S \cup \{e\}$ contains a $3$-configuration, a contradiction to the disjointness of $4^-$- and $3$-configurations.

		It follows from \ref{K=7_second_claim_i''}, \ref{K=7_second_claim_ii''}, \ref{K=7_second_claim_iii''} and \eqref{eqn:J} that
		\begin{equation*}
			|J(\cG)| - |J(\cG')|
			\ge |\cJ|
			\ge 4\binom{r}{t}
			= \binom{r}{t}\left(|\cG| - |\cG'|\right)\,. \qedhere
		\end{equation*}
    \end{proof}
    
    By repeatedly applying \cref{claim:K=7_claim2}, we get a subhypergraph $\cF_3 \subseteq \cF_2$ which is $4^-$-free and satisfies 
	\begin{equation} \label{eqn:F3}
		|J(\cF_2)| - |J(\cF_3)| \ge \binom{r}{t}\left(|\cF_2| - |\cF_3|\right).
	\end{equation}

    \begin{claim} \label{claim:K=7_claim3}
		Let $\cG \subseteq \cF_3$ and suppose there exists a $5^-$-configuration $S$ of $\cG$.
        Then the following holds with $\cG':=\cG \setminus S$.
        \[
        |J(\cG)| - |J(\cG')| \ge \binom{r}{t} \left(|\cG| - |\cG'|\right) \, .
        \]        
    \end{claim}

    \begin{proof}
		Let $\cJ$ be the $t$-shadow of $S$.
	    Then $\cJ \subseteq J(\cG)$ and $|\cJ| = 5\binom{r}{t}$, as a set of size $t$ cannot be in more than one edge of $S$ (otherwise the $5^-$-configuration $S$ would contain a $2$-configuration, a contradiction).
	    
		Next we show that if $T \in \cJ$, then $T \not\in J(\cG')$; for that it is enough to prove that $T$ is not contained in any $\ell$-configuration of $\cG'$ with $\ell \in [3]$.
	    For $\ell=1$, this follows from the fact that any $2$-configuration and any $5^-$-configuration of $\cG$ are edge-disjoint.
	    Similarly, for $\ell=2$ this holds since any $2$-configuration is edge-disjoint of $S$ and thus, if it shares $t$ vertices with $S$, its union with $S$ would give a $7^-$-configuration, a contradiction.
	    Finally, for $\ell=3$, we use that, from \cref{claim:K=7_claim1}, a $3$-configuration cannot be contained in any $4$-configuration of $\cG$.
	    Therefore, we get
        \[
            |J(\cG)| - |J(\cG')| \ge |\cJ| = 5 \binom{r}{t} = \binom{r}{t} \left(|\cG| - |\cG'|\right) \, . \qedhere
        \]   
    \end{proof}

    By repeatedly applying \cref{claim:K=7_claim2}, we get a subhypergraph $\cF_4 \subseteq \cF_3$ which is $5^-$-free and satisfies 
	\begin{equation} \label{eqn:F4}
		|J(\cF_3)| - |J(\cF_4)| \ge \binom{r}{t}\left(|\cF_3| - |\cF_4|\right).
	\end{equation}

	By summing up \eqref{eqn:F2}, \eqref{eqn:F3} and \eqref{eqn:F4}, we get
	\begin{equation*}
		|J(\cF_1)| - |J(\cF_4)| \ge \binom{r}{t}\left(|\cF_1| - |\cF_4|\right).
	\end{equation*}
	Thus, using \Cref{rem:equivalent-condition} and \eqref{eqn:F11},
	\begin{equation*}
		\frac{|\cF_4|}{|J(\cF_4)|}
		\ge \frac{|\cF_1|}{|J(\cF_1)|}\,.
	\end{equation*}
	Notice that $\cF_4$ is $\ell^-$-free for $\ell \in \{2,3,4,5,6\}$ and $7$-free.
	Thus, by \Cref{prop:density}, the limit $\lim_{n \to \infty}n^{-t}f^{(r)}(n;7(r-t)+t,7)$ exists.
\end{proof}

\section{Conclusion}
\label{sec:conclusion}
	Recall that we defined $\pi(r,t,k) := \lim_{n \to \infty}n^{-t}f^{(r)}(n;k(r-t)+t,k)$ (if the limit exists).
	\cref{thm:r>>t} establishes that $\pi(r,t,k)=\frac{1}{t!}\binom{r}{t}^{-1}$ when $k$ is even and $r \ge t + (k^3 \cdot t!)^{1/t}$.
	It would be interesting to determine, for fixed even $k$, what is the smallest $r$ such that $\pi(r,t,k)=\frac{1}{t!}\binom{r}{t}^{-1}$ for all $2 \le t \le r-1$. 
    We remark that the smallest such $r$ is $2$ for $k=2$ and $4$ for $k=4$, as proved in \cite{rodl} and \cite{glockJKKLP_BES_problem}, respectively.

    For general odd $k$ we were not able to prove that the limit $\pi(r,t,k)$ exists, even when $r$ is large. Nevertheless, arguments reminiscent of \Cref{thm:r>>t} yield that if $k$ is odd, $t \ge 2$ and $r$ is sufficiently large with respect to $k$ and $t$, then $\limsup_{n \to \infty}n^{-t}f^{(r)}(n; k(r-t)+t, k) \le \frac{1}{t!} \cdot \frac{2}{2\binom{r}{t}-1}$.
	We briefly sketch the proof idea.

	A \emph{$t$-tight component} is a collection of edges that can be ordered as $\{e_1, \ldots, e_m\}$ so that $e_{i+1}$ shares at least $t$ vertices with one of $e_1, \ldots, e_i$, for each $i \in [m-1]$.
	Given any $k$-free $n$-vertex $r$-graph $\cF$, we let, for $i \in [2]$, $\cF_i$ be the set of the edges of $\cF$ which belong to components of size $i$, and $\cF_3:=\cF \setminus (\cF_1 \cup \cF_2)$.	
	For $i \in [2]$, define $\cG_i$ to be the $t$-shadow of $\cF_i$, and define $\cG_3$ to be the collection of $t$-sets $T$ such that $T \not\in \cG_1 \cup \cG_2$ and there is a unique component in $\cF_3$ that contains a $2$-configuration whose vertex set contains $T$.
    With $\alpha := \frac{1}{2} \cdot \left(2\binom{r}{t}-1\right)$, it is not hard to see that $|\cG_i| \ge \alpha |\cF_i|$ for $i \in [3]$.
    Therefore, since the sets $\cF_1, \cF_2, \cF_3$ partition $\cF$, it follows that $|\cF| = |\cF_1|+|\cF_2|+|\cF_3| \le \frac{1}{\alpha}\left(|\cG_1| + |\cG_2| + |\cG_3|\right) \le \frac{1}{\alpha} \binom{n}{t}$, which implies the desired result.
    
    We suspect this upper bound might be optimal, as this is the case for $k = 3$ (see \cite{glockJKKLP_BES_problem}).

	\begin{remark}
	    The recent work \cite{glockKLPS_recent} mentioned in \cref{remark:recent_paper} shows that, for $k=6$, the smallest $r$ such that $\pi(r,t,6)=\frac{1}{t!}\binom{r}{t}^{-1}$ for all $2 \le t \le r-1$ is $r=4$. 
        Moreover, it shows that the upper bound discussed above is optimal for $k \in \{5,7\}$.
	\end{remark}

\bibliographystyle{amsplain}
\bibliography{references}
      
\end{document}